\documentclass[a4paper]{article}
\usepackage[T1]{fontenc}
\usepackage{lmodern}
\usepackage{amsmath,amssymb}
\usepackage{amsthm}
\usepackage{mathrsfs}
\usepackage[all]{xy}
\usepackage{rotate}
\usepackage{enumitem}
\usepackage{mathtools}
\usepackage[dvipdfmx]{graphics}

\mathtoolsset{showonlyrefs}

\usepackage[dvipdfmx,setpagesize=false]{hyperref}

\usepackage{here}

\theoremstyle{plain}
\newtheorem{theorem}{Theorem}[section]
\newtheorem*{theorem*}{Theorem}
\newtheorem{proposition}[theorem]{Proposition}
\newtheorem{lemma}[theorem]{Lemma}
\newtheorem{corollary}[theorem]{Corollary}

\newtheorem{fact}[theorem]{Fact}

\theoremstyle{definition}
\newtheorem{definition}[theorem]{Definition}

\newtheorem{remark}[theorem]{Remark}

\numberwithin{equation}{subsection}

\newcommand{\define}[1]{\textit{#1}}

\newcommand{\set}[1]{\left\{#1 \right\}}
\newcommand{\esssup}{\textup{ess sup}}

\newcommand{\Ker}{\mathrm{Ker}}

\newcommand{\Hom}{\mathrm{Hom}}
\DeclareMathOperator{\End}{End}

\newcommand{\BasicRing}[1]{\mathbb{#1}}
\newcommand{\ZZ}{\BasicRing{Z}}
\newcommand{\NN}{\BasicRing{N}}

\newcommand{\RR}{\BasicRing{R}}
\newcommand{\CC}{\BasicRing{C}}

\newcommand{\Ann}{\mathrm{Ann}}
\newcommand{\PIdeg}{\mathrm{PI.deg}}

\newcommand{\Len}{\mathrm{Len}}

\newcommand{\alg}[1]{\mathcal{#1}}

\DeclareMathOperator{\gr}{gr}

\newcommand{\Ad}{\mathrm{Ad}}

\newcommand{\lieC}[1]{\mathfrak{#1}}
\newcommand{\lie}[1]{\mathfrak{#1}}

\newcommand{\univ}[1]{\mathcal{U}(\lieC{#1})}

\newcommand{\univcent}[1]{\mathcal{Z}(\lieC{#1})}

\newcommand{\BGGcat}[2]{\mathcal{O}_{\lie{#2}}^{\lie{#1}}}

\newcommand{\sect}{\Gamma}

\newcommand{\zuck}[2]{\Gamma^{#1}_{#2}}
\newcommand{\Dzuck}[3]{R^{#3}\zuck{#1}{#2}}

\newcommand{\regular}{\mathcal{O}}
\newcommand{\rring}[1]{\regular(#1)}

\newcommand{\Variety}{\mathcal{V}}

\newcommand{\Sn}{\mathfrak{S}}
\DeclareMathOperator{\PI}{PI}
\DeclareMathOperator{\sgn}{sgn}

\newcommand{\Hilbert}[1]{\mathcal{#1}}
\newcommand{\toptensor}{\mathbin{\widehat{\otimes}}}
\newcommand{\Jacobson}{\mathcal{J}}
\DeclareMathOperator{\Prim}{Prim}

\begin{document}

\title{Lower semicontinuity of bounded property in the branching problem and sphericity of flag variety}

\author{Masatoshi Kitagawa}

\date{}

\maketitle

\begin{abstract}
	Vinberg--Kimel'fel'd [Funct. Anal. Appl., 1978] established that a quasi-projective normal $G$-variety $X$ is spherical if and only if $G$-modules on the spaces $\sect(X, \mathcal{L})$ of global sections of $G$-equivariant line bundles are multiplicity-free.
	This result was generalized by Kobayashi--Oshima [Adv. Math., 2013] and several researchers to (degenerate) principal series representations of reductive Lie groups.
	The purpose of this short article is to show that the boundedness of the multiplicities in the restrictions of cohomologically induced modules implies the sphericity of some partial flag variety.

	In our previous paper, we reduce the boundedness of the multiplicities to the finiteness of a ring-theoretic invariant $\PIdeg$.
	To show the main result, we discuss the lower semicontinuity of $\PIdeg$ on the space $\mathrm{Prim}(\mathcal{U}(\mathfrak{g}))$ of primitive ideals.
	We also treat the finiteness of the lengths of the restrictions of cohomologically induced modules.
\end{abstract}

\section{Introduction}

In this short article, we deal with the branching problem of reductive Lie groups with bounded multiplicities.
The main result is that the boundedness of the multiplicities in the restrictions of cohomologically induced modules implies the sphericity of some partial flag variety.
To show the result, we discuss the lower semicontinuity of a ring-theoretic invariant $\PIdeg$ on the space of primitive ideals.

Our main result is based on the following classical result.

\begin{fact}[Vinberg--Kimel'fel'd {\cite{ViKi78_spherical}}]\label{intro:fact:VinbergKimelfeld}
	Let $G$ be a connected reductive algebraic group over $\CC$ and $G'$ a connected reductive subgroup of $G$.
	For a partial flag variety $X$ of $G$, the following conditions are equivalent.
	\begin{enumerate}
		\item $X$ is a spherical $G'$-variety, i.e.\ a Borel subgroup of $G'$ has an open orbit in $X$.
		\item For any $G$-equivariant line bundle $\mathcal{L}$ on $X$, the $G'$-module on $\sect(X, \mathcal{L})$ is multiplicity-free.
	\end{enumerate}
\end{fact}

Let $G_\RR$ be a real form of $G$ such that $G'_\RR\coloneq G_\RR \cap G'$ is a real form of $G'$.
Kobayashi--Oshima \cite{KoOs13} established a generalization of Fact \ref{intro:fact:VinbergKimelfeld} for a full flag variety $X$ to infinite-dimensional representations of reductive Lie groups.
They showed that $\dim \Hom_{G'_\RR}(V, V')$ is bounded for smooth admissible representations $V$ and $V'$ of $G_\RR$ and $G'$, respectively, if and only if $X$ is $G'$-spherical.
Several variations of the Kobayashi--Oshima theorem were given in \cite{AiGoMi16,Ta21,Ko22}, replacing $X$ with a partial flag variety and $V$ with degenerate principal series representations realized on a real form of $X$.
In \cite{AiGo21}, similar criteria for the boundedness of multiplicities were given by some condition called $\mathcal{O}$-spherical on nilpotent coadjoint orbits in $\lie{g}^*$.

Spherical actions are strongly related to visible actions on complex manifolds.
The notion of visible actions was introduced by T.\ Kobayashi in \cite{Ko97_multiplicity_free,Ko05}.
Many multiplicity-free representations are obtained from visible actions \cite{Ko13}.
See \cite{Ta22,Ta25} for the relation between visible actions and spherical actions, and references therein for recent studies.

In \cite{KoOs13,Ko22}, the implication 2 $\Rightarrow$ 1 of Fact \ref{intro:fact:VinbergKimelfeld} was also shown.
This is because $V'$ includes finite-dimensional representations when consider the multiplicity $\dim \Hom_{G'_\RR}(V, V')$, and the implication is reduced to the algebraic group case.
In this article, we discuss an analogue of 2 $\Rightarrow$ 1 for cohomologically parabolically induced modules.

The boundedness results in the above are stated for any closed subgroup $H$ of $G$ or $G_\RR$,
and the results for infinite-dimensional representations of reductive Lie groups are obtained from the case of Borel subgroups $H\subset G'$.
Restricting the problem only to reductive Lie groups, we have different types of characterizations \cite{Ki20} of restrictions with bounded multiplicities.
We shall state one of them.

In this article, we deal with $(\lie{g}, K)$-modules.
Although our results hold for unitary representations, smooth admissible representations and analytic representations, we will not state them explicitly.
See Theorem \ref{thm:UniformlyBoundedPIdegInequality} and Remark \ref{rmk:UnitarySmooth}.
For simplicity, suppose that $G_\RR$ and $G'_\RR$ are connected.
Let $\theta$ be a Cartan involution stabilizing $G'_\RR$.
Set $K_\RR \coloneq G_\RR^\theta$ and $K'_\RR \coloneq G'_\RR\cap K_\RR$ and write $K$ (resp.\ $K'$) for the complexification of $K_\RR$ (resp.\ $K'_\RR$).
In \cite{Ki20}, we have shown the following result (see Theorem \ref{thm:UniformlyBoundedPIdegInequality}).

\begin{fact}\label{intro:fact:UniformlyBoundedPIdegInequality}
	Let $V$ be an irreducible $(\lie{g}, K)$-module.
	Then there exists a constant $C>0$ independent of $V$ such that
	\begin{align*}
		\PIdeg((\univ{g}/\Ann_{\univ{g}}(V))^{G'}) &\leq \sup_{W} \dim(\Hom_{\lie{g'}, K'}(V|_{\lie{g'}, K'}, W)) \\
		&\leq C\cdot \PIdeg((\univ{g}/\Ann_{\univ{g}}(V))^{G'}),
	\end{align*}
	where $W$ runs over all irreducible $(\lie{g'}, K')$-modules.
\end{fact}

$\PIdeg$ is a ring-theoretic invariant defined by non-commutative polynomial identities.
Note that for a noetherian $\CC$-algebra $\alg{A}$ of at most countable dimension, $\PIdeg(\alg{A})$ coincides with the maximal dimension of irreducible $\alg{A}$-modules.
See Subsection \ref{subsection:PIdeg}.
Fact \ref{intro:fact:UniformlyBoundedPIdegInequality} says that the boundedness of multiplicities in $V|_{\lie{g'}, K'}$ is controlled by the ring-theoretic invariant $\PIdeg((\univ{g}/\Ann_{\univ{g}}(V))^{G'})$.
The fact is an analogue of the well-known fact that, for an affine smooth $G$-variety $X$, the ring of invariant differential operators is commutative if and only if the coordinate ring $\rring{X}$ is multiplicity-free.

The main tool in this article is the lower semicontinuity of $\PIdeg((\univ{g}/\Ann_{\univ{g}}(V))^{G'})$.
We will show the following result in Theorem \ref{thm:LowerSemicontinuityPIdeg}.

\begin{theorem}\label{intro:thm:LowerSemicontinuity}
	The map $\Prim(\univ{g}) \ni I \mapsto \PIdeg((\univ{g}/I)^{G'}) \in \NN \cup \set{\infty}$
	is lower semicontinuous,
	where $\Prim(\univ{g})$ is the set of primitive ideals with the Jacobson topology.
\end{theorem}

Using Theorem \ref{intro:thm:LowerSemicontinuity}, we will show a variation of 2 $\Rightarrow$ 1 in Fact \ref{intro:fact:VinbergKimelfeld}  for cohomologically parabolically induced modules $\mathcal{R}^j(Z)$.
Let $P$ be a $\theta$-stable parabolic subgroup of $G$ and fix a Levi decomposition $P=LU_P$ such that $L$ is $\theta$-stable and $U_P$ is the unipotent radical of $P$.
Write $K_L$ for the centralizer of the center of $\lie{l}$ in $K$.
For an $(\lie{l}, K_L)$-module $Z$ and $j \in \NN$, set
\begin{align*}
	\mathcal{R}^j(Z) \coloneq \Dzuck{K}{K_L}{j}(\Hom_{\lie{q}}(\univ{g}, Z)_{K_L})
\end{align*}
by letting $\lie{u}_P$ act on $Z$ trivially.
Here $\Dzuck{K}{K_L}{j}$ is the $j$-th Zuckerman derived functor.
Then $\mathcal{R}^j(Z)$ is a $(\lie{g}, K)$-module called a cohomologically induced module.
The following theorem is the main result in this article.
See Theorem \ref{thm:AqLambda} for the details.

\begin{theorem}\label{intro:thm:AqLambda}
	Let $Z$ be an $(\lie{l}, K_L)$-module of finite length and $j \in \NN$.
	If, for sufficiently many characters $\lambda$ of $(\lie{l}, K_L)$, $\mathcal{R}^j(Z\otimes \CC_{\lambda})$ is non-zero and there exists a constant $C > 0$ independent of $\lambda$ such that all multiplicities in $\mathcal{R}^j(Z\otimes \CC_{\lambda})|_{\lie{g'}, K'}$ are bounded by $C$, then $G/P$ is $G'$-spherical.
\end{theorem}

Remark that the converse of Theorem \ref{intro:thm:AqLambda} has been proved in \cite[Theorem 8.9]{Ki20}.
In \cite{Ta25}, it was shown, by using the theory of visible actions, that $\overline{\mathcal{R}^j(Z\otimes \CC_{\lambda})}|_{G'_\RR}$ is multiplicity-free if $\mathcal{R}^j(Z\otimes \CC_{\lambda})$ is unitarizable and $Z$, $j$ and $\lambda$ satisfy some additional conditions.

We also consider almost irreducibility of $V|_{\lie{g'}, K'}$.
When $G'$ acts on the flag variety $G/P$ transitively, induced $G$-modules from irreducible $P$-modules are irreducible as $G'$-modules.
Similar results are known for (cohomologically) parabolically induced modules.
In \cite{Ko11}, a unitary representation $V$ of $G'_\RR$ is said to be almost irreducible if $V$ is a direct sum of irreducible subrepresentations.
See \cite[Section 3]{Ko11} and \cite{Ta25} for almost irreducible restrictions.
Our problem is the converse.
Replacing the sphericity in Theorem \ref{intro:thm:AqLambda} with the transitivity, and the boundedness of multiplicities with that of lengths of $\mathcal{R}^j(Z\otimes \CC_\lambda)|_{\lie{g'}, K'}$,
we will show the following result in Theorem \ref{thm:AqLambdaAlmostIrr}.
The proof will be done exactly parallel to Theorem \ref{intro:thm:AqLambda}.

\begin{theorem}\label{intro:thm:AqLambdaAlmostIrr}
	Let $Z$ be an $(\lie{l}, K_L)$-module of finite length and $j \in \NN$.
	If, for sufficiently many characters $\lambda$ of $(\lie{l}, K_L)$, $\mathcal{R}^j(Z\otimes \CC_{\lambda})$ is non-zero and there exists a constant $C > 0$ independent of $\lambda$ such that the length $\mathcal{R}^j(Z\otimes \CC_{\lambda})|_{\lie{g'}, K'}$ is bounded by $C$, then $G'$ acts on $G/P$ transitively.
\end{theorem}

\subsection*{Notation and convention}

In this paper, any algebra except Lie algebras is unital, associative and over $\CC$.
Any algebraic group is defined over $\CC$ and any Lie algebra is finite dimensional.

We express real Lie groups and their Lie algebras by Roman letters and corresponding German letters with subscript $(\cdot)_\RR$, and express complex Lie groups (or affine algebraic groups) and their Lie algebras by Roman letters and corresponding German letters, respectively.
Similarly, we express the complexification of a real Lie algebra by the same German letter as that of the real form without any subscript.
For example, the Lie algebras of real Lie groups $G_\RR, K_\RR$ and $H_\RR$ are denoted as $\lie{g}_\RR, \lie{k}_\RR$ and $\lie{h}_\RR$ with complexifications $\lie{g}$, $\lie{k}$ and $\lie{h}$, respectively.

For a $\lie{t}$-module $V$ of a commutative Lie algebra $\lie{t}$, we denote by $\Delta(V, \lie{t})$ the set of all non-zero weights in $V$.
For a $G$-set $X$ of a group $G$, we write $X^G$ for the set of all $G$-invariant elements in $X$.
The coordinate ring of an affine variety $X$ is denoted by $\rring{X}$.

\subsection*{Acknowledgement}

This work was supported by JSPS KAKENHI Grant Number JP23K12963.

\section{Primitive spectrum}

In this section, we prepare fundamental results about continuity on the space of primitive ideals in the Jacobson topology.

\subsection{Family of irreducible highest weight modules}\label{subsection:FamilyHighestWeight}

Let $G$ be a connected reductive algebraic group over $\CC$.
Fix a Borel subgroup $B$ of $G$ and its Levi decomposition $B = TU$ with unipotent radical $U$.
Write $\overline{B} = T\overline{U}$ for the opposite Borel subgroup of $B$.

\newcommand{\UMod}[1]{\mathbb{#1}}

For $\lambda \in \lie{t}^*$, we denote by $L(\lambda)$ the irreducible highest weight module (with respect to $\lie{b}$) with the highest weight $\lambda$.
Set
\begin{align*}
	\UMod{L} \coloneq \Hom_{\CC}(\univ{n}, \CC)_{\lie{t}},
\end{align*}
where $(\cdot)_{\lie{t}}$ denotes the functor taking the space of $\lie{t}$-finite vectors.
Then, for each $\lambda \in \lie{t}^*$, we have a natural isomorphism
\begin{align*}
	\Hom_{\univ{\overline{b}}}(\univ{g}, \CC_\lambda)_{\lie{t}} \simeq \Hom_{\CC}(\univ{n}, \CC)_{\lie{t}}
	= \UMod{L}
\end{align*}
as $\lie{n}$-modules.
By this isomorphism, $\UMod{L}$ is equipped with a $\lie{g}$-module structure depending on $\lambda$.
We write $(\pi_{\lambda}, \UMod{L})$ for the $\lie{g}$-module.

\begin{proposition}\label{prop:PolynomialDualVerma}
	For $X \in \univ{g}$, the map $\pi_{\cdot}(X)$ from $\lie{t}^*$ to $\End_{\CC}(\UMod{L})$ is an $\End_{\CC}(\UMod{L})$-valued polynomial function on $\lie{t}^*$, i.e.\ $\pi_{\cdot}(X) \in \rring{\lie{t}^*} \otimes \End_{\CC}(\UMod{L})$.
\end{proposition}

\begin{proof}
	The assertion is standard, so we omit the proof.
\end{proof}

It is well known and clear that $(\pi_\lambda, \UMod{L})$ is isomorphic to the $\lie{t}$-finite dual of the Verma module $\univ{g}\otimes_{\univ{\overline{b}}}\CC_{-\lambda}$.
Hence $L(\lambda)$ is isomorphic to the unique irreducible submodule of $(\pi_\lambda, \UMod{L})$ for each $\lambda \in \lie{t}^*$.
We identify $L(\lambda)$ with the submodule of $(\pi_\lambda, \UMod{L})$.
Define $v_0 \in \UMod{L}$ by
\begin{align*}
	v_0(1) = 1, \quad v_0(\lie{n}\univ{n}) = 0.
\end{align*}
Then $v_0$ is a highest weight vector of all $L(\lambda)$.
In particular, any vector $v \in L(\lambda)$ is written as $v = \pi_\lambda(Y)v_0$ ($Y \in \univ{g}$).

The following lemma is clear.

\begin{lemma}\label{lem:LinearAlgebraNdim}
	Let $V$ be a vector space and $\Phi\subset V^*$ a set of functionals such that $\bigcap_{\varphi \in \Phi} \Ker(\varphi) = 0$.
	Let $W\subset V$ be a subspace and $n \in \NN$.
	Then $\dim(W) < n$ if and only if $\det((\varphi_i(w_j))_{1\leq i,j \leq n}) = 0$ for any $w_1, \ldots, w_n \in W$ and $\varphi_1, \ldots, \varphi_n \in \Phi$.
\end{lemma}

Set $\tau_\lambda(X) \coloneq \pi_\lambda(X)|_{L(\lambda)}$ for $X \in \univ{g}$ and $\lambda \in \lie{t}^*$.

\begin{proposition}\label{prop:ContinuationFiniteDim}
	Let $W\subset \univ{g}$ be a subspace.
	Then the map $\lie{t}^* \ni \lambda \rightarrow \dim(\tau_{\lambda}(W)) \in \NN \cup \set{\infty}$ is lower semicontinuous in the Zariski topology.
\end{proposition}

\begin{proof}
	Let $n \in \NN$ and set $X\coloneq \set{\lambda \in \lie{t}^* : \dim(\tau_{\lambda}(W)) \leq n}$.
	By Lemma \ref{lem:LinearAlgebraNdim}, $X$ is rewritten as
	\begin{align*}
		X = \bigcap_{\substack{w_1, \ldots, w_{n+1} \in W \\ X_1, \ldots, X_{n+1} \in \univ{g} \\ \psi_1, \ldots, \psi_{n+1} \in \UMod{L}^*}} \set{\lambda \in \lie{t}^* : \det (\psi_i(\pi_\lambda(w_j)\pi_\lambda(X_i)v_0)) = 0}.
	\end{align*}
	By Proposition \ref{prop:PolynomialDualVerma}, $X$ is closed.
\end{proof}

\subsection{Primitive spectrum}

The correspondence $\lambda \mapsto \Ann_{\univ{g}}(L(\lambda))$ gives a continuous map from $\lie{t}^*$ to the space of all primitive ideals.
We shall recall the primitive spectrum of a universal enveloping algebra $\univ{g}$ and its topology.

\begin{definition}
	For a (unital associative) $\CC$-algebra $\alg{A}$, we denote by $\Prim(\alg{A})$ the set of all primitive ideals in $\alg{A}$.
	For a subset $T \subset \Prim(\alg{A})$, its closure $\overline{T}$ is defined by
	\begin{align*}
		\overline{T}\coloneq \set{P \in \Prim(\alg{A}): P\supset \bigcap_{I \in T} I}.
	\end{align*}
	The topology on $\Prim(\alg{A})$ is called the Jacobson topology.
\end{definition}

The following proposition is known in \cite[2.2]{So90primespectrum} in more general setting.
We give a proof for the reader's convenience.

\begin{proposition}\label{prop:Continuity}
	The map $\lie{t}^* \ni \lambda \mapsto \Ann_{\univ{g}}(L(\lambda)) \in \Prim(\univ{g})$ is continuous.
\end{proposition}

\begin{proof}
	Let $I$ be a two-sided ideal in $\univ{g}$.
	By Proposition \ref{prop:ContinuationFiniteDim}, for any $X \in I$, the set $\set{\lambda \in \lie{t}^* : X \in \Ann_{\univ{g}}(L(\lambda))} = \set{\lambda \in \lie{t}^*: \tau_{\lambda}(X) = 0}$ is closed in $\lie{t}^*$.
	This implies that $\set{\lambda \in \lie{t}^* : I \subset \Ann_{\univ{g}}(L(\lambda))}$ is closed.
	This shows the continuity of the map.
\end{proof}

The following fact is useful to reduce problem about primitive ideals and $(\lie{g}, K)$-modules to those about highest weight modules.

\begin{fact}[M.\ Duflo {\cite{Du77_primitive_ideal}}]\label{fact:Duflo}
	Let $I$ be a primitive ideal in $\univ{g}$.
	Then there exists an irreducible highest weight module $L$ such that $\Ann_{\univ{g}}(L) = I$.
	In other words, the map in Proposition \ref{prop:Continuity} is surjective.
\end{fact}

\section{Uniformly bounded multiplicities}

In this section, we consider the branching problem of reductive Lie groups.
We will show that the supremum of multiplicities in the restriction $V|_{\lie{g'}, K'}$ is bounded by lower semicontinuous functions on $\Prim(\univ{g})$ from below and above.

\subsection{PI degree}\label{subsection:PIdeg}

We recall the notion of PI degree, which estimates non-commutativity of algebras.
We refer the reader to \cite[Chapter 13]{McRo01_noncommutative}.
We will use PI degree to describe the supremum of multiplicities in a branching law.

\begin{definition}\label{def:PIdeg}
	Let $\alg{A}$ be a (unital associative) $\CC$-algebra.
	For a $\ZZ$-coefficient non-commutative polynomial $f$, we say that $f$ is a \define{polynomial identity} of $\alg{A}$ if $f(X_1, X_2, \ldots, X_r) = 0$ for any $X_i \in \alg{A}$.
	We denote by $\PI(\alg{A})$ the set of all polynomial identities of $\alg{A}$, and set
	\begin{align*}
		\PIdeg(\alg{A}) \coloneq \sup\set{n \in \NN : \PI(\alg{A})\subset \PI(M_n(\CC))},
	\end{align*}
	which is called the \define{PI degree} of $\alg{A}$.
\end{definition}

Clearly, we have inclusions
\begin{align*}
	\PI(M_1(\CC)) \supset \PI(M_2(\CC)) \supset \PI(M_3(\CC)) \supset \cdots.
\end{align*}
It is well known (see Fact \ref{fact:AmitsurLevitzki}) that $\PIdeg(M_n(\CC)) = n$.
Hence the inclusions are strict.

It is clear that if there exists a surjective homomorphism $\alg{A}\rightarrow \alg{B}$ between algebras,
then we have $\PI(\alg{A}) \subset \PI(\alg{B})$ and $\PIdeg(\alg{A}) \geq \PIdeg(\alg{B})$.
Similarly, if there exists an injective homomorphism $\alg{A}\rightarrow \alg{B}$,
then we have $\PI(\alg{A}) \supset \PI(\alg{B})$ and $\PIdeg(\alg{A}) \leq \PIdeg(\alg{B})$.

Under some mild assumption, $\PIdeg(\alg{A})$ coincides with the supremum of the dimensions of irreducible $\alg{A}$-modules.
See \cite[Propositions 2.3 and 2.6]{Ki20}.

\begin{proposition}\label{prop:PIdegMaximumDim}
	Let $\alg{A}$ be a noetherian $\CC$-algebra of at most countable dimension.
	\begin{enumerate}
		\item Let $\set{M_i}_{i \in I}$ be a family of irreducible $\alg{A}$-modules.
		Then one has
		\begin{align*}
			\PIdeg(\alg{A}) \geq \sup\set{\dim(M_i): i \in I}.
		\end{align*}
		The equality holds if $\bigcap_i \Ann_{\alg{A}}(M_i) = 0$.
		\item $\PIdeg(\alg{A})$ coincides with the supremum of the dimensions of irreducible $\alg{A}$-modules.
	\end{enumerate}
\end{proposition}

We shall characterize $\PIdeg(\alg{A})$ by one polynomial $s_{n}$ under a mild assumption.
As a $\ZZ$-coefficient non-commutative polynomial with $n$ indeterminates, define
\begin{align*}
	s_{n}(X_1, \ldots, X_n) \coloneq \sum_{w \in \Sn_n} \sgn(w) X_{w(1)}X_{w(2)}\cdots X_{w(n)},
\end{align*}
where $\Sn_n$ is the symmetric group of degree $n$ and $\sgn$ is the signature character of $\Sn_n$.
Then the polynomials characterize $\PIdeg(M_n(\CC))$.

\begin{fact}[Amitsur--Levitzki {e.g.\ \cite[Proposition 3.2 and Theorem 3.3]{McRo01_noncommutative}}]\label{fact:AmitsurLevitzki}
	For any positive integer $n$, one has
	\begin{align}
		s_{2n}, s_{2n+1}, \ldots \in \PI(M_{n}(\CC)), \quad s_{1}, \ldots s_{2n-1} \not\in \PI(M_{n}(\CC)).
	\end{align}
\end{fact}

Let $\alg{A}$ be a noetherian $\CC$-algebra of at most countable dimension.
We denote by $\Jacobson(\alg{A})$ the Jacobson radical of $\alg{A}$.
By \cite[Corollary 3]{Am56}, $\Jacobson(\alg{A})$ is a nil ideal, i.e.\ any element in $\Jacobson(\alg{A})$ is nilpotent.
By Levitzky's theorem \cite[Theorem 5]{Le50}, any nil ideal in the noetherian ring $\alg{A}$ is nilpotent.
Hence there exists a constant $C \in \NN$ such that $\Jacobson(\alg{A})^C = 0$.

\begin{proposition}\label{prop:CharacterizationPIdeg}
	Let $n \in \NN$.
	Then $\PIdeg(\alg{A}) \leq n$ if and only if $(s_{2n}\cdot X_{2n+1})^C \in \PI(\alg{A})$.
\end{proposition}

\begin{proof}
	Assuming that $m\coloneq \PIdeg(\alg{A}) \leq n$, we shall show $(s_{2n}\cdot X_{2n+1})^C \in \PI(\alg{A})$.
	By Proposition \ref{prop:PIdegMaximumDim}, we have $\PIdeg(\alg{A}/\Jacobson(\alg{A})) = m$
	and there exists a family $\set{(\pi_i, V_i)}_{i\in I}$ of irreducible $\alg{A}$-modules
	such that the homomorphism $\prod_{i} \pi_i \colon \alg{A}/\Jacobson(\alg{A}) \hookrightarrow \prod_i \End_{\CC}(V_i)$ is injective and $\sup_{i \in I} \dim(V_i) = m$.
	This implies that $\PI(\alg{A}/\Jacobson(\alg{A})) = \PI(M_m(\CC)) \ni s_{2n} \cdot X_{2n+1}$.
	Hence we have $(s_{2n}\cdot X_{2n+1})^C \in \PI(\alg{A})$ by $\Jacobson(\alg{A})^C = 0$.

	Assume that $\PIdeg(\alg{A}) > n$.
	By Proposition \ref{prop:PIdegMaximumDim}, there exists an irreducible $\alg{A}$-module $(\pi, V)$ of dimension strictly greater than $n$.
	We shall consider the case of $k\coloneq \dim(V) < \infty$.
	Then 
	\begin{align*}
		s_{2n}(E_{1,1}, E_{1,2}, E_{2,2}, E_{2,3}, \ldots, E_{n,n}, E_{n,n+1})E_{n+1, 1} = E_{1,1}
	\end{align*}
	is an idempotent in $M_k(\CC) \simeq \End_{\CC}(V)$, where $E_{i,j}$'s are the matrix units.
	This shows that $(s_{2n}\cdot X_{2n+1})^C$ is not a polynomial identity of $\End_{\CC}(V)$
	and hence $(s_{2n}\cdot X_{2n+1})^C \not \in \PI(\alg{A})$.

	The proof for $\dim(V) = \infty$ is essentially the same as above by using the Jacobson density theorem, so we omit the proof.
\end{proof}

\subsection{Uniformly bounded multiplicities}\label{subsection:UniformlyBounded}

In \cite{Ki20}, we have given a relation between $\PIdeg$ and the supremum of multiplicities.
We shall recall the results.

Let $G_\RR$ be a connected semisimple Lie group with finite center and $\theta$ a Cartan involution of $G_\RR$.
Set $K_\RR\coloneq G_\RR^\theta$.
Let $G'_\RR$ be a $\theta$-stable connected reductive subgroup of $G_\RR$ and set $K'_\RR\coloneq G'_\RR\cap K_\RR$.
Write $G$ for the inner automorphism group of $\lie{g}$
and $G'$ for the Zariski closure of $\Ad_{\lie{g}}(G'_\RR)$.
Assume that $G'$ is a complexification of $\Ad_{\lie{g}}(G'_\RR)$.
In other words, $\lie{g'}$ is algebraic in $\lie{g}$.
Let $K$ and $K'$ denote the complexifications of $K_\RR$ and $K'_\RR$, respectively.
Then we have a pair $(\lie{g}, K)$ and its subpair $(\lie{g'}, K')$.

Let $B$ be a Borel subgroup of $G$ such that $B'\coloneq G'\cap B$ is a Borel subgroup of $G'$.
Fix a maximal torus $T\subset B$ such that $T'\coloneq T\cap B'$ is a maximal torus of $B'$.
In this subsection, we consider highest weight modules and the BGG category $\BGGcat{}{}$ for these Borel subgroups.
Recall that any irreducible highest weight $\lie{g}$-module is discretely decomposable as a $\lie{g'}$-module.
See \cite{Ko12_generalized_Verma}.

We shall recall the notion of representations with uniformly bounded multiplicities.

\begin{definition}
	Let $\Hilbert{H}$ be a unitary representation of $G'_\RR$ with irreducible decomposition
	\begin{align}
		\Hilbert{H} \simeq \int_{\widehat{G'_\RR}}^\oplus \Hilbert{H}_{\pi}\toptensor \Hilbert{M}_\pi d\mu(\pi), \label{eqn:DirectIntegral}
	\end{align}
	where $\Hilbert{H}_\pi$ is a representation space of $\pi \in \widehat{G'_\RR}$ and $\Hilbert{M}_\pi$ is a Hilbert space.
	We say that $\Hilbert{H}$ has uniformly bounded multiplicities if the essential supremum of $\dim(\Hilbert{M}_\pi)$ is finite.
\end{definition}

For example, a multiplicity-free unitary representation has uniformly bounded multiplicities.

\begin{definition}
	Let $V$ be a $\lie{g'}$-module.
	We say that $V$ has uniformly bounded multiplicities if the supremum of $\dim(\Hom_{\lie{g'}}(V, W))$ for irreducible $\lie{g'}$-modules $W$ is finite.
\end{definition}

When we treat a $(\lie{g'}, K')$-module $V$, it is enough to consider $(\lie{g'}, K')$-modules as $W$.
To deal with the uniform boundedness of multiplicities by $\PIdeg$, we have shown the following result in \cite[Corollary 7.20]{Ki23}.
We denote by $\univcent{g'}$ the center of $\univ{g'}$.

\begin{theorem}\label{thm:FiniteLength}
	Let $V$ be an irreducible $(\lie{g}, K)$-module or an irreducible highest weight module, and $I$ a maximal ideal of $\univcent{g'}$.
	Let $k \in \NN$.
	Then the $(\univ{g'}\otimes \univ{g}^{G'})$-module $V/I^kV$ has finite length
	and the length is bounded by a constant $C$ independent of $I$ and $V$.
\end{theorem}

Using Theorem \ref{thm:FiniteLength}, we can give an upper bound of multiplicities by $\PIdeg$ as follows.
Let $V$ be an irreducible $(\lie{g}, K)$-module and set $I \coloneq \Ann_{\univ{g}}(V)$.
Suppose that $n \coloneq \PIdeg((\univ{g}/I)^{G'}) < \infty$.
Then the dimension of any irreducible $(\univ{g}/I)^{G'}$-module is less than or equal to $n$ by Proposition \ref{prop:PIdegMaximumDim}.
By Theorem \ref{thm:FiniteLength}, we obtain
\begin{align*}
	\dim \Hom_{\lie{g'}, K'}(V|_{\lie{g'}, K'}, V') &= \dim \Hom_{\lie{g'}, K'}(V/\Ann_{\univ{g'}}(V')V, V') \\
	&\leq C\cdot \PIdeg((\univ{g}/I)^{G'})
\end{align*}
for any irreducible $(\lie{g'}, K')$-module $V'$.
We have also given a lower bound of the supremum of multiplicities by $\PIdeg$ in the previous paper.
The following theorems are in \cite[Theorems 1.2 and 7.11]{Ki20}.

\begin{theorem}\label{thm:UniformlyBoundedPIdeg}
	Let $V$ be a non-zero $(\lie{g}, K)$-module of finite length or an object in the BGG category $\BGGcat{}{}$.
	Then the following conditions are equivalent.
	\begin{enumerate}
		\item $V|_{\lie{g'}}$ has uniformly bounded multiplicities.
		\item $\PIdeg((\univ{g}/\Ann_{\univ{g}}(V))^{G'}) < \infty$.
	\end{enumerate}
	If $V$ is a unitarizable $(\lie{g}, K)$-module and $\overline{V}$ is the Hilbert completion, the above conditions are equivalent to
	\begin{enumerate}\setcounter{enumi}{2}
		\item $\overline{V}|_{G'_\RR}$ has uniformly bounded multiplicities.
	\end{enumerate}
\end{theorem}

\begin{theorem}\label{thm:UniformlyBoundedPIdegInequality}
	Let $V$ be an irreducible $(\lie{g}, K)$-module or an irreducible highest weight module.
	Then there exists a constant $C>0$ independent of $V$ such that
	\begin{align*}
		\PIdeg((\univ{g}/\Ann_{\univ{g}}(V))^{G'}) &\leq \esssup\set{\dim(\Hilbert{M}_\pi):\pi \in \widehat{G'_\RR}} \\
		&\leq \sup_{W} \dim(\Hom_{\lie{g'}}(V|_{\lie{g'}}, W)) \\
		&\leq C\cdot \PIdeg((\univ{g}/\Ann_{\univ{g}}(V))^{G'}),
	\end{align*}
	where $\Hilbert{M}_\pi$ is defined in Definition \ref{eqn:DirectIntegral}
	and $W$ runs over all irreducible $\lie{g'}$-modules.
\end{theorem}

\begin{remark}\label{rmk:UnitarySmooth}
	If $V$ is not unitarizable, the inequality holds without the term $\esssup\set{\dim(\Hilbert{M}_\pi):\pi \in \widehat{G'_\RR}}$.
	Although a similar inequality holds for smooth admissible representations of moderate growth, we omit it in this paper for simplicity.
\end{remark}

By Theorem \ref{thm:UniformlyBoundedPIdegInequality}, the boundedness of multiplicities does not depend on the choice of the categories (i.e.\ $(\lie{g}, K)$-modules, smooth representations or unitary representations).
Hereafter, we concentrate on $(\lie{g}, K)$-modules and $\PIdeg$.

\subsection{Lower semicontinuity of \texorpdfstring{$\PIdeg$}{PIdeg}}

In this subsection, we show the lower semicontinuity of $\PIdeg$.
To show the result, we need to bound the constant $C$ in Proposition \ref{prop:CharacterizationPIdeg}.

Write $W_{G'}$ for the Weyl group of $G'$.
For a $\lie{g'}$-module $V$ and a character $\chi$ of $\univcent{g'}$, we set
\begin{align*}
	V_{\chi} \coloneq \set{v \in V: \Ker(\chi)^k v = 0 \text{ for some }k\in \NN}.
\end{align*}
$V_{\chi}$ is called the $\chi$ primary component of $V$.
If $V$ is the direct sum of all primary components, we say that $V$ has the primary decomposition.
The following lemma is proved in \cite[Theorem 7.133 and its proof]{KnVo95_cohomological_induction}.

\begin{lemma}\label{lem:GeneralizedInfinitesimalCharHighestWeightModule}
	Let $V$ be a $\lie{g'}$-module with an infinitesimal character, and $F$ a finite-dimensional $G'$-module.
	Then $V\otimes F$ has the primary decomposition and $\Ker(\chi)^{|W_{G'}|}$ annihilates $(V\otimes F)_{\chi}$ for any character $\chi$ of $\univcent{g'}$.
\end{lemma}

We consider the algebra $(\univ{g}/I)^{G'}$ for a primitive ideal $I \subset \univ{g}$.
Recall that $\Jacobson((\univ{g}/I)^{G'})$ is the Jacobson radical.
As we have seen before Proposition \ref{prop:CharacterizationPIdeg}, there exists a constant $C \in \NN$ such that $\Jacobson((\univ{g}/I)^{G'})^C = 0$.
We shall show that the constant $C$ can be independent of $I$.

\begin{theorem}\label{thm:UniformNilpotency}
	There exists a constant $C \in \NN$ such that, for any primitive ideal $I$ of $\univ{g}$, $\Jacobson((\univ{g}/I)^{G'})^C = 0$.
\end{theorem}

\begin{proof}
	Take a Borel subgroup $B$ of $G$ such that $B'\coloneq B\cap G'$ is a Borel subgroup of $G'$.
	Fix a Levi decomposition $B=TU$ such that $B'=(T\cap G')(U\cap G')$ is a Levi decomposition.
	Set $T'\coloneq T\cap G'$ and $U'\coloneq U\cap G'$.

	Let $I$ be a primitive ideal in $\univ{g}$.
	By Fact \ref{fact:Duflo}, there exists an irreducible highest weight $\lie{g}$-module $L$ with respect to $\lie{b}$ such that $\Ann_{\univ{g}}(L) = I$.
	Then $L$ is a faithful module of $(\univ{g}/I)^{G'}$.

	Take a highest weight vector $v_0 \in L$.
	Then $\univ{g'}v_0$ is a highest weight $\lie{g'}$-module and hence has finite length.
	Take an irreducible submodule $L_0$ of $\univ{g'}v_0$.
	Then we have $L = \univ{g}L_0$.
	Since $\univ{g}$ is completely reducible as a $G'$-module, by Lemma \ref{lem:GeneralizedInfinitesimalCharHighestWeightModule}, $L|_{\lie{g'}}$ has the primary decomposition and $\Ker(\chi)^{|W_{G'}|}$ annihilates $L_\chi$ for any character $\chi$ of $\univcent{g'}$.
	In other words, we have $L_{\chi} \simeq L/\Ker(\chi)^{|W_{G'}|} L$.
	
	Take a constant $C$ in Theorem \ref{thm:FiniteLength} for $k = |W_{G'}|$.
	Let $\chi$ be a character of $\univcent{g'}$ such that $L_{\chi}\neq 0$.
	To show the assertion, it is enough to see that $\Jacobson((\univ{g}/I)^{G'})^{C}$ annihilates $L_{\chi}$.
	Let $W$ be an irreducible subquotient of the $(\univ{g'}\otimes \univ{g}^{G'})$-module $L_{\chi}$.
	Then $W$ has a $\lie{b}'$-eigenvector and hence $W|_{\lie{g'}}$ is a sum of some copies of an irreducible highest weight module $L'$.
	This implies that there exists an irreducible $\univ{g}^{G'}$-module $L''$ such that $W \simeq L' \boxtimes L''$.
	Since $\Jacobson((\univ{g}/I)^{G'})$ annihilates $W \simeq L'\boxtimes L''$, we have $\Jacobson((\univ{g}/I)^{G'})^{C}L_{\chi} = 0$.
	We have shown the theorem.
\end{proof}

\begin{remark}
	It is easy to see from the proof that if $L|_{\lie{g'}}$ is completely reducible, then $(\univ{g}/I)^{G'}$ is semiprimitive and hence $\Jacobson((\univ{g}/I)^{G'}) = 0$.
	If there exists an irreducible unitary representation $V$ of $G_\RR$ such that $I = \Ann_{\univ{g}}(V^\infty)$ and $\PIdeg((\univ{g}/I)^{G'}) < \infty$, then $(\univ{g}/I)^{G'}$ is semiprimitive.
	This is proved by using the direct integral decomposition \eqref{eqn:DirectIntegral}.
	See \cite[Corollary 10.25]{Ki20} for the details.
\end{remark}

We shall show the lower semicontinuity of $\PIdeg$.

\begin{theorem}\label{thm:LowerSemicontinuityPIdeg}
	The map $\Prim(\univ{g}) \ni I \mapsto \PIdeg((\univ{g}/I)^{G'}) \in \NN \cup \set{\infty}$
	is lower semicontinuous.
\end{theorem}

\begin{proof}
	Let $n \in \NN$ and set $X \coloneq \set{I \in \Prim(\univ{g}): \PIdeg((\univ{g}/I)^{G'}) \leq n}$.
	We shall show that $X$ is closed.
	Put $I_0 \coloneq \bigcap_{I \in X} I$.
	Then the natural homomorphism
	\begin{align*}
		(\univ{g}/I_0)^{G'} \rightarrow \prod_{I \in X} (\univ{g}/I)^{G'}
	\end{align*}
	is injective.
	This implies
	\begin{align*}
		\PI((\univ{g}/I_0)^{G'}) \supset \bigcap_{I \in X} \PI((\univ{g}/I)^{G'}).
	\end{align*}
	Take a constant $C$ as in Theorem \ref{thm:UniformNilpotency}.
	By Proposition \ref{prop:CharacterizationPIdeg} and Theorem \ref{thm:UniformNilpotency},
	we have $(s_{2n}\cdot X_{2n+1})^C \in \PI((\univ{g}/I_0)^{G'})$.

	Let $I'$ be a primitive ideal of $\univ{g}$ containing $I_0$.
	Then $(\univ{g}/I')^{G'}$ is a quotient algebra of $(\univ{g}/I_0)^{G'}$ since the $G'$-action on $\univ{g}$ is completely reducible.
	Hence we have $(s_{2n}\cdot X_{2n+1})^C \in \PI((\univ{g}/I)^{G'})$.
	By Proposition \ref{prop:CharacterizationPIdeg}, we obtain $\PIdeg((\univ{g}/I)^{G'})\leq n$.
	
	Therefore we have $X = \set{I \in \Prim(\univ{g}): I_0\subset I}$.
	This implies that $X$ is closed.
\end{proof}

\subsection{Spherical v.s.\ uniformly bounded multiplicities}\label{subsection:SphericalVS}

As an application of Theorem \ref{thm:LowerSemicontinuityPIdeg}, we shall show that the uniform boundedness of multiplicities implies the sphericity of a flag variety.

Let $B$ be a Borel subgroup of $G$ with Levi decomposition $B = TU$,
and $P$ a standard parabolic subgroup of $G$ with Levi decomposition $P = LU_P$.
Then $B_L \coloneq L\cap B$ is a Borel subgroup of $L$.
Set $U_L \coloneq U\cap L$.
Write $\rho$ for half the sum of all roots in $\Delta(\lie{u}, \lie{t})$,
and $W_G$ for the Weyl group of $G$.

The following lemma asserts that one generalized Verma module knows the sphericity of $G/P$.
The result has been proved in \cite[Proposition 6.12 and Theorem 8.3]{Ki20}.

\begin{lemma}\label{lem:BoundedSphericalAlgebraic}
	Let $F$ be a non-zero $\lie{l}$-module.
	Set $I \coloneq \Ann_{\univ{g}}(\univ{g}\otimes_{\univ{p}}F)$, letting $\lie{u}_P$ act on $F$ trivially.
	If $\PIdeg((\univ{g}/I)^{G'}) < \infty$, then $G/P$ is $G'$-spherical.
	If $F$ is finite dimensional, the converse also holds.
\end{lemma}

\begin{lemma}\label{lem:SphericalAndLevi}
	Let $Q$ be a parabolic subgroup of $G$.
	Assume that a Levi subgroup of $Q$ is conjugate to $L$ by an inner automorphism.
	Then $G/Q$ is $G'$-spherical if and only if so is $G/P$.
\end{lemma}

\begin{proof}
	The assertion has been proved in several ways \cite[Theorem 1.3]{AvPe14}, \cite[2.2]{AiGo21}.
	The proof is based on the results in which $G/P$ is $G'$-spherical if and only if the corresponding Richardson orbit has a good properties (e.g.\ coisotropic actions, $\mathcal{O}$-sphericity).
	We can give a proof using Lemma \ref{lem:BoundedSphericalAlgebraic}.

	We may assume that $P$ and $Q$ have the same Levi subgroup $L$.
	Let $U_Q$ be the unipotent radical of $Q$.
	Write $\rho_P$ (resp.\ $\rho_Q$) for half the sum of roots in $\Delta(\lie{u}_P, \lie{t})$ (resp.\ $\Delta(\lie{u}_Q, \lie{t})$).
	By \cite[Corollar 15.27]{Ja83}, we have $\Ann_{\univ{g}}(\univ{g}\otimes_{\univ{p}}\CC_{-\rho_P}) = \Ann_{\univ{g}}(\univ{g}\otimes_{\univ{q}}\CC_{-\rho_Q})$.
	This and Lemma \ref{lem:BoundedSphericalAlgebraic} show the assertion.
\end{proof}

One generalized Verma module has enough information about the sphericity of $G/P$.
In contrast to this, one of induced representations on $G/P$ (e.g.\ $\rring{G/P}$) has less information about $G/P$.
To recover structure on $G/P$, we need to consider a family of $\lie{g}$-modules.
Let $\lie{a}_L$ be the center of $\lie{l}$.
Then we have $\lie{a}_L\subset \lie{t}$.

\begin{lemma}\label{lem:GeneralForm}
	Let $S\subset \Prim(\univ{g})$.
	For each $I \in S$, take a representative $\lambda_I \in \lie{t}^*$ of the infinitesimal character of $I$.
	Suppose that $\overline{\set{\lambda_I: I \in S}}$ contains $\lie{a}_L^* + \mu_0$ for some $\mu_0 \in \lie{t}^*$.
	If $\sup_{I \in S}\PIdeg((\univ{g}/I)^{G'}) < \infty$, then $G/P$ is $G'$-spherical.
\end{lemma}

\begin{proof}
	Assume that $\sup_{I \in S}\PIdeg((\univ{g}/I)^{G'}) < \infty$.
	For each $I \in S$, take an irreducible highest weight module $L(\lambda'_I)$ such that $\Ann_{\univ{g}}(L(\lambda'_I)) = I$ using Fact \ref{fact:Duflo}.
	Then we have $\lambda'_I = w_I(\lambda_I) - \rho$ for some $w_I \in W_G$.
	Since $W_G$ is a finite set, replacing $S$ with its subset, we may assume that
	\begin{align*}
		w\coloneq w_I = w_{I'} \quad (\forall I, I' \in S),\\
		\overline{\set{w(\lambda_I) - \rho : I \in S}} = w(\lie{a}_L^* + \mu_0) - \rho.
	\end{align*}
	By assumption, Proposition \ref{prop:Continuity} and Theorem \ref{thm:LowerSemicontinuityPIdeg}, we have
	\begin{align*}
		\PIdeg((\univ{g}/\Ann_{\univ{g}}(L(\lambda)))) < \infty
	\end{align*}
	for any $\lambda \in w(\lie{a}_L^* + \mu_0) - \rho$.

	Set $L_0\coloneq Z_G(w(\lie{a}_L))$ and $P_0\coloneq L_0U$.
	Then $P_0$ is a standard parabolic subgroup of $G$ and $L_0$ is a Levi subgroup of $G$ conjugate with $L$ under an inner automorphism.
	It is well known that for any irreducible highest weight $\lie{l}_0$-module $F$,
	the generalized Verma module $\univ{g}\otimes_{\univ{q_\mathrm{0}}}(F\otimes \CC_{\mu})$ is irreducible for generic $\mu \in w(\lie{a}^*_L)$.
	This implies that there exists $\lambda \in w(\lie{a}_L^* + \mu_0) - \rho$ such that $L(\lambda)$ is isomorphic to an irreducible generalized Verma module $\univ{g}\otimes_{\univ{q_\mathrm{0}}}(F\otimes \CC_{\mu})$.
	By Lemma \ref{lem:BoundedSphericalAlgebraic}, $G/P_0$ is $G'$-spherical since $\PIdeg((\univ{g}/\Ann_{\univ{g}}(L(\lambda)))) < \infty$.
	By Lemma \ref{lem:SphericalAndLevi}, this shows that $G/P$ is $G'$-spherical.
\end{proof}

We shall state the main result in this section.
Recall the notion of cohomologically induced modules.
Suppose that $P$ and $L$ are $\theta$-stable.
Write $K_L$ for the centralizer of the center of $\lie{l}$ in $K$.
For an $(\lie{l}, K_L)$-module $Z$, set
\begin{align*}
	\mathcal{R}^j(Z) \coloneq \Dzuck{K}{K_L}{j}(\Hom_{\lie{q}}(\univ{g}, Z)_{K_L}).
\end{align*}
Here $\Dzuck{K}{K_L}{j}$ is the $j$-th Zuckerman derived functor, i.e.\ $\Dzuck{K}{K_L}{0}(V) = V_K$ for a $(\lie{g}, K_L)$-module $V$.
Then $\mathcal{R}^j(Z)$ is a $(\lie{g}, K)$-module called a cohomologically induced module.

If $Z$ has finite length, so does $\mathcal{R}^j(Z)$.
If $Z$ has an infinitesimal character $[\lambda] \in \lie{t}^*/W_{G}$, then $\mathcal{R}^j(Z)$ has the infinitesimal character $[\lambda - \rho_{\lie{u}_P}]$, where $\rho_{\lie{u}_P}$ is half the sum of all roots in $\Delta(\lie{u}_P, \lie{t})$.
See \cite[Theorem 0.46 and Proposition 0.48]{KnVo95_cohomological_induction}.
We refer the reader to \cite{KnVo95_cohomological_induction} for more detailed structure of $\mathcal{R}^j(Z)$.
Note that our $\mathcal{R}^j(Z)$ is the unnormalized version written as ${}^{\mathrm{u}}\mathcal{R}^j(Z)$ in \cite[11.71]{KnVo95_cohomological_induction}.

The following theorem follows from Lemma \ref{lem:GeneralForm} immediately.

\begin{theorem}\label{thm:AqLambda}
	Let $Z$ be an $(\lie{l}, K_L)$-module of finite length, and $\set{\lambda_i}_{i \in I}$ a family of characters of $(\lie{l}, K_L)$.
	Fix $j \in \NN$.
	Suppose that $\mathcal{R}^j(Z\otimes \CC_{\lambda_i})$ is non-zero for any $i \in I$ and $\set{\lambda_i: i \in I}$ is Zariski dense in $\lie{a}^*_L$.
	If there exists a constant $C > 0$ such that, for any $i \in I$, all multiplicities in $\mathcal{R}^j(Z\otimes \CC_{\lambda_i})|_{\lie{g'}, K'}$ are bounded by $C$ from above, then $G/P$ is $G'$-spherical.
\end{theorem}

\begin{remark}
	As we have seen in Theorem \ref{thm:UniformlyBoundedPIdegInequality}, if $\mathcal{R}^j(Z\otimes \CC_{\lambda_i})$ is unitarizable, we can replace $\mathcal{R}^j(Z\otimes \CC_{\lambda_i})$ with its Hilbert completion.
\end{remark}

\begin{remark}
	The converse of Theorem \ref{thm:AqLambda} has been proved in \cite[Theorem 8.9]{Ki20}.
	Under some stronger conditions, it has been proved in \cite{Ta25} that $\mathcal{R}^j(Z\otimes \CC_{\lambda_i})|_{\lie{g'}, K'}$ is multiplicity-free.
\end{remark}

\section{Almost irreducible restriction}

In the previous section, we considered the relation between the sphericity of $G/P$ and the uniform boundedness of multiplicities in the restriction $V|_{\lie{g'}, K'}$.
In this section, we will show a similar result replacing the sphericity with the transitivity, and the uniform boundedness of multiplicities with that of lengths of $V|_{\lie{g'}, K'}$.

\subsection{Almost irreducibility}

Retain the notation in Subsections \ref{subsection:UniformlyBounded} and \ref{subsection:SphericalVS}.
A unitary representation $V$ of $G'_\RR$ is said to be almost irreducible in \cite{Ko11} if $V$ has finite length.
It is rare that, for an irreducible unitary representation $V$ of $G_\RR$, the restriction $V|_{G'_\RR}$ is a finite direct sum of irreducible subrepresentations.
When $G'$ acts on the flag variety $G/P$ transitively, induced $G$-modules from characters of $P$ are irreducible as $G'$-modules.
Similar results are known for (cohomologically) parabolically induced modules.
See \cite[Section 3]{Ko11}.
In this section, we will show the converse of these results similarly to Theorem \ref{thm:AqLambda} for the uniform boundedness.

In this subsection, we give a characterization of the almost irreducibility by the algebra $\univ{g}^{G'}$.
This is an analogue of Theorem \ref{thm:UniformlyBoundedPIdegInequality}.
For a module $V$, we denote by $\Len(V)$ the length of $V$.
If $V$ does not have finite length, $\Len(V) = \infty$.

\begin{lemma}\label{lem:LengthAndDim}
	Let $\alg{A}$ be a $\CC$-algebra of at most countable dimension, $V$ and $W$ $\alg{A}$-modules of finite length.
	Then $\dim(\Hom_{\alg{A}}(V, W))$ is less than or equal to $\Len(V)\Len(W)$.
\end{lemma}

\begin{proof}
	If $V$ and $W$ are irreducible, the assertion is clear from Schur's lemma.
	It is easy to see that the assertion for general $V$ and $W$ follows by induction on the lengths of $V$ and $W$.
\end{proof}

The following lemma is an analogue of Theorem \ref{thm:UniformlyBoundedPIdegInequality}.

\begin{lemma}\label{lem:BoundLenDim}
	Let $V$ be an irreducible $(\lie{g}, K)$-module or an irreducible highest weight $\lie{g}$-module with respect to $\lie{b}$.
	Set $I\coloneq \univ{g}/\Ann_{\univ{g}}(V)$.
	Then there exists a constant $C > 0$ independent of $V$ such that
	\begin{align*}
		\dim((\univ{g}/I)^{G'})^{1/2}
		\leq \Len(V|_{\lie{g'}}) \leq C\cdot \dim((\univ{g}/I)^{G'})^2.
	\end{align*}
\end{lemma}

\begin{proof}
	The first inequality follows from Lemma \ref{lem:LengthAndDim} since $(\univ{g}/I)^{G'}$ is embedded in $\End_{\lie{g'}}(V)$.
	We shall show the second inequality.
	
	If $\dim(\univcent{g'}/\Ann_{\univcent{g'}}(V)) = \infty$, then every terms in the desired inequality are infinite.
	Hence it is enough to show the assertion assuming $\dim(\univcent{g'}/\Ann_{\univcent{g'}}(V)) < \infty$.
	Then $V|_{\lie{g'}}$ has the primary decomposition
	\begin{align*}
		S \coloneq \set{\chi \in \Hom_{\mathrm{alg}}(\univcent{g}, \CC) : (V|_{\lie{g'}})_{\chi} \neq 0}, \quad V|_{\lie{g'}} = \bigoplus_{\chi \in S} (V|_{\lie{g'}})_{\chi}.
	\end{align*}
	
	Take $\chi \in S$.
	Let $W\subset (V|_{\lie{g'}})_{\chi}$ be a finitely generated $\lie{g'}$-submodule.
	Since $W$ is a finitely generated $(\lie{g'}, K')$-module (or locally $\lie{b'}$-finite module) with the generalized infinitesimal character $\chi$,
	$W$ has finite length.
	See \cite[Theorem 4.2.1]{Wa88_real_reductive_I}.
	In particular, $V|_{\lie{g'}}$ has an irreducible submodule $V_0$.
	Then we have $V = \univ{g}V_0$.

	Since the $G'$-action on $\univ{g}$ is completely reducible, by Lemma \ref{lem:GeneralizedInfinitesimalCharHighestWeightModule}, we have
	\begin{align*}
		V|_{\lie{g'}} = \bigoplus_{\chi \in S} V/\Ker(\chi)^{|W_{G'}|} V,
	\end{align*}
	and by Theorem \ref{thm:FiniteLength}, the length of $(\univ{g'}\otimes \univ{g}^{G'})$-module $V$ is bounded by $|S|\cdot C'$ for some constant $C' > 0$ independent of $V$.
	Recall that, as we have seen in the proof of Theorem \ref{thm:UniformNilpotency}, any irreducible $\univ{g'}\otimes (\univ{g}/I)^{G'}$-subquotient of $V$ is of the form $L'\boxtimes L''$, where $L'$ is an irreducible $\lie{g'}$-module and $L''$ is an irreducible $(\univ{g}/I)^{G'}$-module.
	Since the dimension of any irreducible $(\univ{g}/I)^{G'}$-module is bounded by $\dim((\univ{g}/I)^{G'})$, we have
	\begin{align*}
		|S| &\leq \dim(\univcent{g'}/\Ann_{\univcent{g'}}(V)) \leq \dim((\univ{g}/I)^{G'}) \\
		\Len(V|_{\lie{g'}}) &\leq |S| \cdot C' \cdot \dim((\univ{g}/I)^{G'}) \leq C'\cdot \dim((\univ{g}/I)^{G'})^2. \qedhere
	\end{align*}
\end{proof}

\begin{lemma}\label{lem:OrbitTransitive}
	Let $O$ be the closure of a nilpotent coadjoint orbit in $\lie{g}^*$.
	$\rring{O}^{G'} = \CC$ if and only if $G'$ has an open orbit in $O$.
\end{lemma}

\begin{proof}
	`If' part is clear, so we shall show the converse.
	Assume $\rring{O}^{G'} = \CC$.
	Let $K(X)$ denote the field of rational functions on an irreducible algebraic variety $X$.
	Let $\widetilde{O}$ be the normalization of $O$.
	By \cite[Theorem 1.2.4]{Lo09}, $K(\widetilde{O})^{G'}( = K(O)^{G'})$ coincides with the field of fractions of $\rring{\widetilde{O}}^{G'}$.
	It is well known that $\rring{\widetilde{O}}^{G'}$ is the integral closure of $\rring{O}^{G'}$ in $K(O)^{G'}$.
	In fact, the Reynolds operator from $\rring{O}[t]$ to $\rring{O}^{G'}[t]$ is an $\rring{O}^{G'}[t]$-module homomorphism and the integrality follows from this fact.
	Hence we have $K(O)^{G'}= K(\widetilde{O})^{G'} = \CC$ and this shows that $G'$ has an open orbit in $O$ by Rosenlicht's theorem \cite[Section 2.3]{ViPo94}.
\end{proof}

The following corollary is an easy consequence of Lemma \ref{lem:OrbitTransitive}.

\begin{corollary}\label{cor:SmallOrbit}
	Let $O$ be a nilpotent coadjoint orbit in $\lie{g}^*$ and $O'$ a $G$-orbit in $\overline{O}$.
	If $G'$ has an open orbit in $O$, then so does in $O'$.
\end{corollary}

For an ideal $I$ of a coordinate ring $\rring{X}$ on an affine algebraic variety $X$,
we denote by $\Variety(I)$ the subvariety of $X$ determined by the ideal $I$.
For an ideal $I$ of $\univ{g}$, we denote by $\gr I$ the associated graded ideal of $I$ by the standard filtration on $\univ{g}$.
The following theorem is an analogue of Theorem \ref{thm:UniformlyBoundedPIdeg}.

\begin{theorem}\label{thm:AlmostIrr}
	Let $V$ be a non-zero $(\lie{g}, K)$-module of finite length or an object in the BGG category $\BGGcat{}{}$.
	Set $O\coloneq \Variety(\gr \Ann_{\univ{g}}(V)) \subset \lie{g}^*$.
	Then the following conditions are equivalent.
	\begin{enumerate}
		\item $V|_{\lie{g'}}$ has finite length.
		\item $\dim((\univ{g}/\Ann_{\univ{g}}(V))^{G'}) < \infty$.
		\item $\rring{O}^{G'} = \CC$.
		\item $G'$ has an open orbit in each irreducible component of $O$.
	\end{enumerate}
	If $V$ is a unitarizable $(\lie{g}, K)$-module with Hilbert completion $\Hilbert{H}$,
	then the above conditions are equivalent to
	\begin{enumerate}\setcounter{enumi}{4}
		\item $\Hilbert{H}|_{G'_\RR}$ is a finite direct sum of irreducible subrepresentations.
	\end{enumerate}
\end{theorem}

\begin{proof}
	First, we remark that if $\rring{O}^{G'}$ is finite dimensional, then $\rring{O}^{G'} = \CC$.
	This follows from that $\rring{O}^{G'}$ is reduced and a graded algebra with the $0$-th term $\CC$.

	We have shown 1 $\Leftrightarrow$ 2 in Lemma \ref{lem:BoundLenDim},
	and 3 $\Leftrightarrow$ 4 in Lemma \ref{lem:OrbitTransitive}.
	Although $O$ is irreducible in Lemma \ref{lem:OrbitTransitive}, the equivalence 3 $\Leftrightarrow$ 4 is proved by considering irreducible components of $O$.

	Assume the condition 2, i.e.\ $\dim((\univ{g}/\Ann_{\univ{g}}(V))^{G'}) < \infty$.
	Since $G'$ is reductive, we have
	\begin{align}
		\gr((\univ{g}/\Ann_{\univ{g}}(V))^{G'}) \simeq (S(\lie{g})/\gr \Ann_{\univ{g}}(V))^{G'} \label{eqn:IsomGr}
	\end{align}
	as vector spaces.
	Hence $\rring{O}^{G'} \simeq (S(\lie{g})/\sqrt{\gr \Ann_{\univ{g}}(V)})^{G'}$ is finite dimensional.
	This implies the condition 3.

	We shall show the converse assuming $\rring{O}^{G'} = \CC$.
	By $\rring{O}^{G'} \simeq (S(\lie{g})/\sqrt{\gr \Ann_{\univ{g}}(V)})^{G'}$, there exists a filtration of $(S(\lie{g})/\gr \Ann_{\univ{g}}(V))^{G'}$ such that its associated graded algebra is a finitely generated $\rring{O}^{G'}$-module.
	Hence $(\univ{g}/\Ann_{\univ{g}}(V))^{G'}$ is finite dimensional by \eqref{eqn:IsomGr}.

	The last assertion is clear from the Harish-Chandra admissibility theorem.
\end{proof}

\subsection{Transitivity v.s.\ almost irreducibility}

We use the notation in Subsection \ref{subsection:SphericalVS}.
Then $P = LU_P$ is a parabolic subgroup of $G$.
We shall show an analogue of Theorem \ref{thm:AqLambda} for the almost irreducibility.

Write $\mu\colon T^*(G/P)\rightarrow \lie{g}^*$ for the moment map.
Then $\mu$ is the multiplication map under the identification $T^*(G/P)\simeq G\times_P (\lie{g}/\lie{p})^*$.
It is well known that $\mu(T^*(G/P))$ is the closure of one nilpotent coadjoint orbit called a Richardson orbit.
Note that $T^*(G/P)$ and $\mu(T^*(G/P))$ have the same dimension.
See \cite[Chapter 7 and Theorem 7.1.1]{CoMc93} for the Richardson orbit.

\begin{lemma}\label{lem:NilpotentToFlag}
	$G'$ acts on $G/P$ transitively if and only if $G'$ has an open orbit in $\mu(T^*(G/P))$.
\end{lemma}

\begin{proof}
	Assume that $G'$ acts on $G/P$ transitively.
	Then $G'/G'\cap P \simeq G/P$ is a flag variety of $G'$.
	This implies that $G'$ has an open orbit in $T^*(G/P) \simeq T^*(G'/G'\cap P)$ and hence in $\mu(T^*(G/P))$.
	In fact, $G'\cap P$ has an open orbit in $(\lie{g'}/\lie{g'}\cap \lie{p})^*$ (see \cite[Theorem 7.1.1]{CoMc93}).

	To show the converse, assume that $G'$ has an open orbit in $\mu(T^*(G/P))$.
	Since $T^*(G/P)$ and $\mu(T^*(G/P))$ have the same dimension, $G'$ has an open orbit in $T^*(G/P)$ and hence in $G/P$.
	Let $X$ be the open $G'$-orbit in $G/P$.
	Write $p\colon \lie{g}^* \rightarrow (\lie{g'})^*$ for the restriction.
	Then $p\circ \mu\colon T^*(G/P)\rightarrow (\lie{g'})^*$ is the moment map for the $G'$-action on $T^*(G/P)$.
	By assumption, $p(\mu(T^*(G/P)))$ contains an open dense $G'$-orbit and $\set{0}$.
	This implies that $\overline{p(\mu(T^*X))} = \overline{p(\mu(T^*(G/P)))}$ is a closure of a nilpotent coadjoint orbit in $(\lie{g'})^*$.

	By \cite[Theorem 8.17]{Ti11}, the rank of $X$ is zero (i.e.\ $B'$-semiinvariant rational functions on $X$ are only constant functions), and this shows that $X$ is a flag variety of $G'$ by \cite[Proposition 10.1]{Ti11}.
	Therefore, $X$ is open and closed in $G/P$ and hence we obtain $X = G/P$.
\end{proof}

We shall state analogues of Lemmas \ref{lem:SphericalAndLevi} and \ref{lem:BoundedSphericalAlgebraic}.

\begin{corollary}\label{cor:TransitiveAndLevi}
	Let $Q$ be a parabolic subgroup of $G$.
	Assume that a Levi subgroup of $Q$ is conjugate to $L$ by an inner automorphism.
	Then $G'$ acts on $G/P$ transitively if and only if so does on $G/Q$.
\end{corollary}

\begin{proof}
	It is well known that $G/P$ and $G/Q$ has the same Richardson orbit, more precisely, the Richardson orbit of $P$ is determined only by the conjugacy class of Levi subgroups of $P$ \cite[Theorem 7.1.3]{CoMc93}.
	Hence the assertion follows immediately from Lemma \ref{lem:NilpotentToFlag}.
\end{proof}

\begin{corollary}\label{cor:FiniteLenTransitive}
	Let $F$ be a non-zero $\lie{l}$-module.
	Set $I \coloneq \Ann_{\univ{g}}(\univ{g}\otimes_{\univ{p}}F)$, letting $\lie{u}_P$ act on $F$ trivially.
	If $\dim (\univ{g}/I)^{G'} < \infty$, then $G'$ acts on $G/P$ transitively.
	If $F$ is finite dimensional, the converse also holds.
\end{corollary}

\begin{proof}
	We shall show that $\Variety(\gr I)$ contains $\mu(T^*(G/P))$.
	Since $\univ{g}$ is a free right $\univ{p}$-module, we have $I \subset \univ{g}\Ann_{\univ{p}}(F)$.
	Hence we have
	\begin{align*}
		\gr I \subset \bigcap_{g \in G} \Ad(g) (S(\lie{g})\gr \Ann_{\univ{p}}(F))
		\subset \bigcap_{g \in G} \Ad(g) (S(\lie{g})\lie{p}).
	\end{align*}
	This shows that $\Variety(\gr I)$ contains $\mu(T^*(G/P)) = \Ad^*(G)(\lie{g}/\lie{p})^*$.
	The first assertion follows from this, Theorem \ref{thm:AlmostIrr} and Lemma \ref{lem:NilpotentToFlag}.
	The second assertion follows from the first and the well-known fact $\Variety(\gr I) = \mu(T^*(G/P))$ if $\dim(F) < \infty$.
\end{proof}

The following results are analogues of Lemma \ref{lem:GeneralForm} and Theorem \ref{thm:AqLambda}.
The proofs are essentially the same as the previous ones, so we omit the proofs.
Note that we have shown the lower semicontinuity of $\dim (\univ{g}/\Ann_{\univ{g}}(L(\lambda)))^{G'}$ on $\lie{t}^*$ in Proposition \ref{prop:ContinuationFiniteDim}.

\begin{lemma}\label{lem:GeneralFormAlmostIrr}
	Let $S\subset \Prim(\univ{g})$.
	For each $I \in S$, take a representative $\lambda_I \in \lie{t}^*$ of the infinitesimal character of $I$.
	Suppose that $\overline{\set{\lambda_I: I \in S}}$ contains $\lie{a}_L^* + \mu_0$ for some $\mu_0 \in \lie{t}^*$.
	If $\sup_{I \in S} \dim (\univ{g}/I)^{G'} < \infty$, then $G'$ acts on $G/P$ transitively.
\end{lemma}

\begin{theorem}\label{thm:AqLambdaAlmostIrr}
	Let $Z$ be an $(\lie{l}, K_L)$-module of finite length, and $\set{\lambda_i}_{i \in I}$ a family of characters of $(\lie{l}, K_L)$.
	Fix $j \in \NN$.
	Suppose that $\mathcal{R}^j(Z\otimes \CC_{\lambda_i})$ is non-zero for any $i \in I$ and $\set{\lambda_i: i \in I}$ is Zariski dense in $\lie{a}^*_L$.
	If there exists a constant $C > 0$ such that, for any $i \in I$, the length of $\mathcal{R}^j(Z\otimes \CC_{\lambda_i})|_{\lie{g'}, K'}$ is bounded by $C$ from above, then $G'$ acts on $G/P$ transitively.
\end{theorem}

\begin{remark}
	As we have seen in Theorem \ref{thm:AlmostIrr}, if $\mathcal{R}^j(Z\otimes \CC_{\lambda_i})$ is unitarizable, we can replace $\mathcal{R}^j(Z\otimes \CC_{\lambda_i})$ with its Hilbert completion.
\end{remark}

\begin{remark}
	The converse of Theorem \ref{thm:AlmostIrr} has been proved in \cite[Theorem 3.5]{Ko11}.
	See also Theorem \ref{thm:AlmostIrr} and Corollary \ref{cor:FiniteLenTransitive}.
\end{remark}

\bibliographystyle{abbrv}


\def\cprime{$'$} \def\cprime{$'$}

\end{document}